\documentclass[11pt,a4paper]{article}

\usepackage{amsfonts,amsmath,amsfonts,graphicx,graphics,epsfig}
\usepackage{enumerate}

\newtheorem{theorem}{Theorem}

\newtheorem{lemma}{Lemma}

\newtheorem{proposition}{Proposition}

\newenvironment{proof}[1][Proof]{\noindent\textit{#1.} }{\ \rule{0.5em}{0.5em}}

\def\R{\mathbb{R}}
\def\L{\mathbb{L}}

\begin{document}

\author{Cristina Butucea and Pierre Vandekerkhove \\
Universit\'{e} Paris-Est Marne-la-Vall\'{e}e}
\title{Semiparametric mixtures of symmetric distributions }
\maketitle

\begin{abstract}
We consider in this paper the semiparametric mixture of two
distributions  equal  up to a shift parameter. The model is said
to be semiparametric  in the sense that
the mixed distribution is not supposed to belong to a parametric family.
In order to insure the identifiability  of the model it is assumed that the mixed
distribution is symmetric, the model being then defined by the mixing
proportion,
two location parameters, and the probability density function of the mixed
distribution. We propose a new class of $M$-estimators of these parameters
based on a Fourier approach, and prove that they are
$\sqrt{n}$-consistent under mild regularity conditions.  Their
finite-sample properties are illustrated by a Monte Carlo study and a
benchmark real dataset is also studied  with our method.

\end{abstract}
\vspace{3cm}
\noindent{\bf AMS 2000 subject classifications.} Primary 62G05, 62G20; secondary 62E10. \\
\noindent{\bf Key words and phrases.}  Asymptotic normality, consistency, contrast
estimators, Fourier transform, identifiability, inverse problem, semiparametric, two-component mixture model.

\bigskip

\newpage 

\section{Introduction}

The probability density functions (pdf) of $d$-variate multicomponent  mixture models  are defined by
\begin{eqnarray}\label{Mixture}
g(x)=\sum_{i=1}^k \lambda_i f_i(x),\quad \quad x\in \R^d,
\end{eqnarray}
where the unknown proportions $\lambda_i$  ($\lambda_i\geq 0$ and $\sum_{i=1}^k \lambda_i=1$) and the unknown pdf $f_i$ are to be estimated. Generally the $f_i$'s are supposed to belong to a parametric family of density functions turning the inference problem for model (\ref{Mixture})
into a purely parametric estimation problem. There exists an extensive literature on this subject including the monographs of
Everitt and Hand (1981), Titterington
{\it et al.} (1985) or McLachlan and Peel (2000), which provide a good overview of the existing  methods in this case such as maximum likelihood, minimum chi-square, moments method, Bayesian approaches etc. Note that the estimation of the number of components $k$
in model (\ref{Mixture}) may also be a crucial issue leading to various rates of convergence for maximum likelihood estimators, as discussed
by Chen (1995). In that case, the selection model is an important topic, see for example Dacunha-Castelle \& Gassiat  (1999), Lemdani \&  Pons (1999),  and Leroux (1992).
In addition the choice of a parametric family for the $f_i$'s may be difficult when few informations are known from each subpopulations.
However, model (\ref{Mixture}) is generally nonparametrically nonidentifiable without additionnal assumptions.  This is no longer true when training data
are available from each subpopulation; see for example Cerrito (1992),  Hall (1981), Lancaster \& Imbens (1996),   Murray \& Titterington (1978), and Qin (1999).
Hall and Zhou (2003) first considered the case where no parametric assumptions
are made about the $f_i$'s involved in model (\ref{Mixture}). These authors looked at $d$-variate mixtures of two distributions, each having independent components, and proved that, under mild regularity conditions, their model is identifiable when $d\geq 3$. They propose in addition
$\sqrt{n}$-consistent estimators of the $2d$ univariate marginal cumulative distribution functions and the mixing proportion. Even if model
(\ref{Mixture}) is not nonparametrically identifiable there exists for $d=1$ and $k\geq 2$, many real data sets in the statistical literature  for which such a model is used under parametric assumptions on the $f_i$'s, such as the Old Faithfull dataset, see Azzalini \& Bowman (1990), which corresponds to time measurement  (in minute) between eruptions of the Old Faithfull geyser in Yellowstone National Park, USA. Another famous  example  deals with average amounts of precipitation (rainfall) in inches for United States cities (from the Statistical abstract of the United States, 1975; see McNeil (1977).These data sets  are  both included in the \texttt{R}  statistical package.\\
To model from a semiparametric point of view this type of data  ($d=1$ and $k\geq 2$), Bordes, Mottelet \& Vandekerkhove (2006) (in abreviate BMV) and Hunter, Wang \& Hettmansperger (2007) (in abreviate HWH) proposed jointly  to consider i.i.d. sample data $(X_{1},...,X_{n})$ drawn from a common pdf  $g$ satisfying
\begin{eqnarray}\label{BMV}
g(x)=\sum_{i=1}^k \lambda_i f(x-\mu_i), \quad\quad x\in \mathbb R,
\end{eqnarray}
where $\mu_i \in \R$, $\lambda_i \geq 0$ for all $i \in \{1,...,k\}$ such that $\sum_{i=1}^k \lambda_i = 1$ and $f$ is an unknown pdf. When $f$ is supposed to be symmetric about zero, that is $f(x)=f(-x)$ for all $x\in \R$, the above authors proposed $M$-estimation methods based on the cumulative distribution function (cdf) in order to estimate  separately the Euclidean and functional part of model (\ref{BMV}).
The crucial part of their work deals with the  identifiability of model (\ref{BMV}) under the simple symmetry assumption on $f$. Their basic results are established in BMV, Theorem 2.1 and HWH, Theorem 1, 2 and Corollary 1.
The mixed density $g$ in (\ref{BMV}) can also be seen as the density of i.i.d. observations $X_i$ in a convolution model:
\begin{equation}\label{deconv}
X_i = Z_i+\varepsilon_i, \quad i=1,...,n,
\end{equation}
where $Z_i$'s are i.i.d. with common pdf $f$ and independent of i.i.d. errors $\varepsilon_i$'s with discrete law such that $P(\varepsilon = \mu_i) = \lambda_i$, for $i=1,...,k$. Previous results mean that, if $k$ is known and $f$ is supposed to be symmetric about 0, then we can identify the law of the errors and esimate nonparametrically the pdf $f$. Let us notice that the mixture problem in (\ref{BMV}) and the deconvolution problem in (\ref{deconv}) are the same.  They are both an inverse problem with unknown operator (i.e. convolution with an unknown law having support on $k$ unknown points).
In particular when $k=2$, $\lambda_1:=p_0$ and $(\mu_1,\mu_2):=(\alpha_0,\beta_0)$, according to Theorem 2.1. in BMV,  such a model is  identifiable  if  the Euclidean parameter $\theta_0:=(p_0,\alpha_0,\beta_0)\in [0,1/2)\times{ \mathbb R}^2\setminus \Delta$, where $\Delta=\left\{(x,x); \, x\in \mathbb R\right\}$ and the mixed density $f$ is symmetric about~0. When $k=2$, BMV  prove, under mild conditions, that both the Euclidean parameter and the cumulative distribution function of $f$ of model (\ref{BMV}) are estimated almost surely  at the rate $n^{-1/4+\alpha}$, for all $\alpha >0$ (see Theorem 3.3 and 3.4). When $k=2$ or 3, HWH  prove under mild conditions, the strong  consistency of their estimator, and establish, under very technical conditions, its asymptotic normality (see Theorems 3 and 4 therein).

In this paper we propose to investigate a new estimation method.
Let us first recall that  BMV propose an iterative procedure to invert the operator and a contrast which is based on the cdf $G$ and the symmetry of the underlying unknown pdf $f$.
HWH introduce a contrast based on the cdf of the observations $G$ and
estimate the euclidean parameter using the symmetry property of the unknown pdf $f$.
Here, we use Fourier analysis to invert the operator and see that under identifiability assumptions the inverse problem is well posed. Then we construct a contrast based on characteristic functions of our data which allows to estimate $\theta$ when $f$ is symmetric. This contrast is a functional of $g$ which is estimated by a U-statistic of order 2 at parametric rate under very mild smoothness assumption on $f$ (Sobolev smoothness larger than 1/4).
Our procedure is easier to deal with and allows to get a central limit theorem for the estimator of $\theta$ under much simpler conditions than those of Theorem 4 in HWH. Moreover, we define a kernel estimator of the pdf $f$ and prove that it attains the same nonparametric rate as in the direct problem of density estimation. The inverse problem does not affect the pointwise rate of convergence of the density estimator. Our estimators and convergence results generalize to the mixture model with $k\geq 3$ components, as soon as the model verifies identifiability assumptions. Such assumptions are known for $k=3$ only, see Corollary 1 in HWH.

The paper is organized as follows: in Section 2 we propose a contrast
function based on a Fourier transform of the dataset pdf and derive our
estimation method; in Section 3 we present our main asymptotic result
which concern the $\sqrt{n}$-rate  of convergence for the Euclidean part
of the parameter and show that the classical nonparametric rate of
convergence  is achieved for our inverse Fourier  nonparametric estimator;
Section 4 is dedicated to auxiliary results and proofs; in Section 5 we
propose a Monte Carlo study of our estimators on several  simulated
examples and implement our method on a real dataset which deals with the
average amounts of precipitation (rainfall) in inches for United States
cities, see  McNeil (1977).

\section{Estimation procedure}

We observe $X_1,\ldots,X_n$ independent, identically distributed random variables having
common pdf $g$ in the model
\begin{eqnarray}\label{model}
g(x)=p_0 f(x-\alpha_0 )+(1-p_0 )f(x-\beta_0 ),  \quad\quad x\in \mathbb R,
\end{eqnarray}
where $\theta_0:=(p_0,\alpha_0,\beta_0)$ denotes the unknown value of the Euclidean  parameter
and $f \in \L_2$ is unknown, symmetric pdf in a large nonparametric class of functions.

For identifiability reasons, let $\theta_0$ belong to a compact set
$\Theta \subset (0,1/2) \times \R^2\setminus \Delta$. Therefore, there are positive $P_*, \, P$, which are smaller than 1/2, such that $p_0 \in [P_*,P]$.

Note that in case $p_0 = 0$ we can still identify $\beta_0$ but not $\alpha_0$. As this case reduces to the estimation of the location of an unknown symmetric pdf $f$ as in Beran (1978), we do not consider this case further on.

\bigskip

From now on, we denote by $f^*(u) = \int_{\R} e^{ixu} f(x) dx$ the Fourier transform and recall
that if $f^* \in \L_1$ we have the inversion formula $f(x) = (2\pi)^{-1} \int_{\R} e^{-iux} f^*(u) du$.

\bigskip

Let us denote  $M(\theta,u):=pe^{iu\alpha}+(1-p)e^{iu\beta}$,
for all $\theta \in \Theta$ and $u \in \mathbb R$,
and see that it cannot be $0$ as soon
as $p\neq 1/2$.  It is enough to notice that $(1-2P)^2\leq|M(\theta,u)|^2\leq 1$
for all $(u,\theta)\in \mathbb R\times \Theta$.

The contrast uses the symmetry of the underlying, unknown pdf $f$. For the first time
in the literature of mixture models,
we relate the symmetry of $f$ to the fact that its Fourier transform has no imaginary part.
More precisely, in model (\ref{model})
$$
g^*(u) = (p_0 e^{iu\alpha_0 }+(1-p_0 )e^{iu\beta_0 })f^{\ast }(u)
 = M(\theta_0,u) f^*(u), \quad\quad u\in \mathbb R.
$$
When $f$ is supposed to be symmetric about $0$, we can hope that
$Im (g^*(u)/M(\theta,u))=0$, for all $u\in \mathbb R$,  if and only  if $\theta=\theta_0$.
This basic result is formally stated in the following theorem.
\begin{theorem}\label{theoImag}
Consider model (\ref{BMV}) with $f$ symmetric about 0 and $\theta_0\in \Theta$. Then
we have $Im\left(g^{\ast }/M(\theta,\cdot)\right)=0$ for some $\theta\in \Theta$ if
and only if $\theta=\theta_0$.
\end{theorem}
\begin{proof}
Notice that for all $\theta\in \Theta$ such that $Im\left(g^{\ast }/M(\theta,\cdot)\right)=0$ we explicitly have
$$
Im\left(\frac{g^{\ast}(u)}{M(\theta,u)}\right)
=Im\left(f^\ast(u)\frac{M(\theta_0,u)}{M(\theta,u)} \right)
= \frac{f^\ast(u)}{|  M(\theta,u) |^2}  Im\left((M(\theta_0,u)\bar M(\theta,u))\right)=0,
$$
for all $u\in \mathbb R$. As $f^*(0)=1$, we get that $Im(M(\theta_0,\cdot)\bar{M}(\theta,\cdot))$ is null in a neighborhood of $0$ which leads, following the proof of Theorem 2.1 in BMV,
to the wanted result  $\theta=\theta_0$.
\end{proof}

Assuming $g^\ast $ known we can recover
the true value of the Euclidean  parameter  by minimizing the discrepancy measure $S$ defined by
\begin{eqnarray}\label{Stheta}
S(\theta):=\int_{\mathbb R} \left(Im\left( \frac{g^{\ast }(u)}{M(\theta,u)}\right)\right)^2dW(u),
\quad\quad \theta\in \Theta,
\end{eqnarray}
where $W$ is  a Lebesgue-absolutely continuous probability measure supported by $\R$.

Note that we can also write
$$
    S(\theta)
    = \int_{\R} \left[\frac 1{2i} \left( \frac{g^{\ast }(u)}{M(\theta,u)} - \frac{\bar{g^{\ast }}(u)}{\bar{M}(\theta,u)} \right)\right]^2 dW(u).
$$
From now on, $\bar z$ denotes the complex conjugate of $z$.

\begin{proposition}\label{propcontrast}
The function $S$ in (\ref{Stheta}) is a contrast function, i.e. for all $\theta \in \Theta$, $S(\theta)\geq 0$ and $S(\theta)=0$ if and only if $\theta=\theta_0$.
\end{proposition}
\begin{proof}
The Fourier transform $f^{\ast}$ being continuous, the same holds for $Im\left(\frac{g^{\ast }}{M(\theta,\cdot)}\right)$. By Theorem 1, if $\theta\neq \theta_0$ there exists $u_0\in \mathbb R$  such that  $Im\left(\frac{g^{\ast }(u_0)}{M(\theta,u_0)}\right)\neq 0$, and there exists $\varepsilon >0$ and $\gamma>0$ such that
$Im\left(\frac{g^{\ast }(u)}{M(\theta,u)}\right)>\varepsilon$ on $[u_0-\gamma,u_0+\gamma]$. It follows that
\begin{eqnarray*}
S(\theta)\geq \varepsilon ^2 \int_{u_0-\gamma}^{u_0+\gamma} dW(u)>0.
\end{eqnarray*}
Otherwise if $\theta=\theta_0$ it is straightforward to check that $S(\theta)=0$.
\end{proof}

\noindent{\it Discussion.}  We point out that  basic results similar to Theorem \ref{theoImag} and Proposition \ref{propcontrast}, can be established for model (\ref{BMV}) when $k=3$ under sufficient identiability conditions.  Indeed, in that case, it is enough to replace $\theta$ by $(\lambda_1, \lambda_2, \mu_1, \mu_2, \mu_3)^T$ and $M(\theta,u)$ by $\sum_{j=1}^3 \lambda_j e^{iu\mu_j}$ and check that  the analog of Theorem \ref{theoImag} can be established following the Proof of Lemma A. 1, under conditions provided in Corollary 1, in HWH.  Finally, similar estimators to those in Sections \ref{Contrmini} and \ref{nonp} and asymptotic results like those established in Section \ref{mainresult}  for  $k=2$, can be established with a little extra work for  $k=3$.


\subsection{Contrast minimization for the Euclidean parameter}\label{Contrmini}

Let the estimator of $\theta_0$ be the following M-estimator
\begin{eqnarray}\label{estimateur}
\hat \theta_n=\arg\min_{\theta\in \Theta} S_n(\theta),
\end{eqnarray}
where $S_n(\theta)$, depending on some parameter $h>0$ (small with $n$),
is the following estimator of $S(\theta)$
\begin{eqnarray}\label{estimS}
S_n(\theta) = \frac{-1}{4 n(n-1)}
\int_{|u|\leq 1/h} \sum_{j \not= k, j,k=1}^n
\left( \frac{e^{i uX_k}}{M(\theta,u)} - \frac{e^{-i uX_k}}{M(\theta,-u)} \right)
\left( \frac{e^{i uX_j}}{M(\theta,u)} - \frac{e^{-i uX_j}}{M(\theta,-u)} \right) dW(u).
\end{eqnarray}
The estimator $S_n(\theta)$ is inspired by kernel estimators of quadratic functional
of the pdf $f$ as previously studied in Butucea (2007). It is written here in the
Fourier domain. It is known that by removing the diagonal terms in the double sum
({\it i.e.}  taking $j \not = k$) the bias is reduced with respect to the estimator where we plug
an estimator of $g^*$ into $S(\theta)$.

Let us denote by
\begin{eqnarray*}
Z_k(\theta,u)&:=& \frac{e^{i uX_k}}{M(\theta,u)} - \frac{e^{-i uX_k}}{M(\theta,-u)}, \\
 J(\theta,u)&:=&\frac{g^{\ast }(u)}{M(\theta,u)} - \frac{{g^{\ast }}(-u)}{{M}(\theta,-u)}.
\end{eqnarray*}
Then it is easy to see that
\begin{eqnarray*}
S_n(\theta) & = & \frac {-1}{4n(n-1)} \sum_{j \not= k, j,k=1}^n \int_{|u|\leq 1/h}
Z_k(\theta,u) Z_j(\theta,u) dW(u),\\
S(\theta) & = & -\frac 14 \int_{\R} J^2(\theta,u) dW(u),
\end{eqnarray*}
and that $E[Z_k(\theta,u)] = J(\theta,u)$.

\subsection{Kernel based nonparametric estimator}\label{nonp}

After estimating the Euclidean parameter, we want to estimate the nonparametric function $f$. We
suggest to use cross-validation for a kernel estimator as follows.
We denote by $\hat \theta_{n, -k}$ the leave-one-out estimator of $\theta_0$, which uses the sample without the $k$-th observation. Then we plug this in the classical nonparametric kernel estimator, whenever the unknown $\theta_0$ is required.
This procedure gives, in Fourier domain,
\begin{equation}\label{kernelfour}
 f^*_{n} (u) = \frac 1n \sum_{k=1}^n \frac{K^*(b_n u)e^{iuX_k}}{M(\hat \theta_{n,-k},u)},
\end{equation}
where $K$ the kernel ($\int K = 1$ and $K \in \mathbb{L}_2$) and $b_n$ the bandwidth are properly chosen.
Note that $G_n^*(u):=K^*(b_n u)/M(\hat \theta_{n,-k},u)$ is in $\L_1$ and $\L_2$ and has an inverse
Fourier transform which we denote by $G_n(u/b_n)/b_n$. Therefore, the estimator of $f$ is
\begin{equation}\label{kernelest}
f_n(x) = \frac 1{nb_n} \sum_{k=1}^n G_n\left( \frac{x-X_k}{b_n}\right).
\end{equation}

It is important to notice at this step, that the estimator $f_n$ is obtained by inversion
of a nonparametric kernel estimator 
\begin{eqnarray}\label{gkern}
g_n(x) = \frac 1{n b_n} \sum_{k=1}^n K\left( \frac{x-X_k}{b_n}\right),
\end{eqnarray}{
with kernel $K$ and bandwidth $b_n$. The inversion is done in Fourier domain with the estimated 
$\hat \theta_{n,-k}$ instead of the true $\theta_0$:
$$
f_n^*(u) = \frac{g_n^*(u)}{M(\hat \theta_{n,-k},u)}.
$$
When dealing with  the rain fall dataset studied  in Section 4, we propose to consider, as in BMV,  the version $\tilde f_n$ of  the estimator $ f_n(x)$ (which has a negative part due to the small number of observations) defined by 
\begin{eqnarray}\label{tildf}
\tilde f_n(x)=\frac{ f_n(x) \mathbb{I}_{f_n(x)\geq 0}}{\int_\R f_n(x) \mathbb{I}_{ f_n(x)\geq 0}} .
\end{eqnarray}

\section{Main results}\label{mainresult}

Let us state first several assumptions.

\noindent {\bf Assumption A} Let $W: \R \to \R^+$ be a cumulative distribution function of some random variable which admits finite absolute moments up to the third order:
$$
\int_{\R} (1+ |u|+u^2 + |u|^3) dW(u) <\infty.
$$

\noindent {\bf Assumption B} We assume that the underlying probability density $f$ belongs
to a ball of radius $L>0$ in the Sobolev space of functions having smoothness $\beta >0$:
$$
W(\beta,L) = \left\{ f: \R \to \R_+: \int f=1, \int |f^*(u)|^2 |u|^{2\beta} du \leq L\right\},
$$
where $f^*$ denotes the Fourier transform of the function $f$.

The weight function $W$ has been introduced for integrability of our estimator $S_n(\theta)$ of the criterium $S(\theta)$ and its derivatives with respect to $\theta$. It is completely arbitrary and it may help compute numerically the values of our integrals by Monte-Carlo simulation, but it slightly reduces the asymptotic efficiency of $\hat \theta_n$. We could have used integrals with respect to the Lebesgue measure for highest efficiency of $\hat \theta_n$, but this would require stronger assumptions of smoothness and moments for the unknown probability density function $f$.

\begin{proposition}\label{convS_n}
For each $\theta\in \Theta$, the empirical contrast function $S_n(\cdot)$ defined in (\ref{estimS})
with $h\to 0$ when $n\to \infty$, is such that
$$
\sup_{f\in W(\beta, L)} \sup_{\theta \in \Theta} E\left[ \left( S_n(\theta) - S(\theta)\right)^2\right]
\leq \frac {L^2}{(1-2P)^4}h^{4\beta} + \frac{1}{(1-2P)^2 n} +\frac{1}{(1-2P)^4n^2},
$$
as $n \to \infty$.
\end{proposition}


An easy consequence of the Theorem is that
$|S_n(\theta)-E(S_n(\theta))| = O_P(n^{-1/2})$ as
$n \to \infty$.

Moreover, if we choose $h = o(1) n^{-1/(4\beta)}$ the squared
bias of $S_n(\theta)$ is infinitely smaller when compared to
its variance. So the mean squared error converges at $n^{-1}$ rate
as soon as $\beta >1/4$.

\begin{theorem}\label{cons}
The estimator $\hat \theta_n$  defined in (\ref{estimateur})  converges in probability to the true value of the
Euclidean parameter $\theta_0$ as $n\rightarrow \infty$.
\end{theorem}

\begin{theorem}\label{NA}
The estimator $\hat \theta_n$  defined in (\ref{estimateur}) with $h\to 0$ such that
$h = o(1) n^{-1/(4\beta)}$ is asymptotically normally distributed:
$$
\sqrt{n} (\hat \theta_n - \theta_0) \stackrel{d}{\to} N(0,\Sigma), \mbox{ as } n\to \infty,
$$
where $\Sigma = \mathcal{I}^{-1} V \mathcal{I}$,
$\mathcal{I}=\mathcal{I}(\theta_0) = -\frac 12 \int_\R \dot J(\theta_0,u) \dot J ^\top (\theta_0,u) dW(u)$,
$V=\frac 14 E(U_1(\theta_0) U_1^\top (\theta_0))$
and $U_1(\theta_0) = \int_\R Z_1(\theta_0,u) \dot J(\theta_0,u) dW(u)$.
\end{theorem}

The next theorem gives the upper bounds for the rate of convergence of the nonparametric estimator $f_n$ of $f$, at some fixed point $x$, over Sobolev classes of functions. The main message of the theorem is that, if $\beta >1/2$ then the nonparametric rates for density estimation are reached, provided a correct choice of parameters $h$ and $b_n$. This might seem surprising, but it is again related to the fact that the inverse problem under consideration is well posed and the estimation of the Euclidean  parameter $\theta_0$ does not affect the nonparametric rate for estimating $f$.

\begin{theorem}\label{nonparam}
Let the estimator $\hat \theta_n$ of $\theta$
be defined in (\ref{estimateur}) and $f_n(x)$ the estimator of $f(x)$
at some fixed point $x \in \R$ in (\ref{kernelest}), with $h = o(1) n^{-1/(4\beta)}$, $b_n=c n^{-(\beta-1/2)/(2\beta)}$ for some $c>0$ and a kernel $K$ in $\mathbb{L}_1$ and in $\mathbb{L}_2$ with Fourier transform $K^*$ having support included in $\{u: |u|\geq 1\}$.

If $\beta >1/2$,
$$
\limsup_{n\to \infty} \sup_{f\in W(\beta,L)} \sup_{\theta_0 \in \Theta}
E_{\theta_0,f}\left[ n^{-\frac{2\beta -1}{2\beta}} |f_n(x)-f(x)|^2\right]
\leq C,
$$
for some constant $C < \infty$ which depends on $\beta, \, L, \, P$ and on $\int K^2$.
\end{theorem}

We can choose an arbitrary point $\theta \in \Theta$ and write
$$
\sup_{f\in W(\beta,L)} \sup_{\theta_0 \in \Theta}
E_{\theta_0,f}\left[ n^{-\frac{2\beta -1}{2\beta}} |f_n(x)-f(x)|^2\right]
\geq \sup_{f\in W(\beta,L)}
E_{\theta,f}\left[ n^{-\frac{2\beta -1}{2\beta}} |f_n(x)-f(x)|^2\right]
$$
The lower bounds are known in the case of density estimation from direct observations, see for example  results for more general Besov classes of functions in H\"ardle {\it et al.} (1998). They generalize easily to our case, with fixed $\theta$.

\section{Simulations}

We implement our method and study its behaviour on samples of size $n=100$. The mean behaviour of our estimator $\hat \theta_n$ of $\theta_0$ is calculated by replicating $M=100$ times the same experiment.
We considered that the underlying symmetric density is either Gaussian, Cauchy or Laplace. We give the mean value of the estimated parameter and its standard deviation in Tables 1, 3 and 4, respectively.  We also plot the nonparametric estimator of the underlying density as compared to the true, in Figure~1.

We see that smaller is $p$, smaller is the standard deviation of $\hat \beta_n$. This is indeed intuitively clear, as $1-p$ which is larger represents the fraction of data sampled from the second population or else the amount of information about the population which is located at $\beta$. 

We note that the previous estimation methods based on the distribution function require usually finite moments up to some order. These methods cannot deal with the Cauchy density that we consider here, see Table \ref{cauchy}. Indeed, our method is based on Fourier transform, which is fast decreasing in this case.
We also consider non smooth Laplace density (or double exponential), see Table \ref{laplace}. Its Fourier transform is slowly decreasing, but we chose the weight function $w(x) = e^{-|x|}$ in order to deal with this problem. Therefore, all integrals have relatively small support of integration and the computation is fast enough.

\begin{table}[hptb!]\label{gauss}
\begin{center}
\begin{tabular}{cccc}
$n$ & $(p_0,\alpha_0,\beta_0)$ & Empirical means & Standard deviations\\
\hline
100 & (0.05, -1, 2) & (0.0808, -1.0398, 2.0181) & (0.0477, 0.3038, 0.1354)\\
100 & (0.10, -1, 2) & (0.1205, -1.0433, 1.9990) & (0.0478, 0.2829, 0.1569)\\
100 & (0.15, -1, 2) & (0.1609, -0.9874, 2.0093) & (0.0406, 0.2964, 0.1455)\\
100 & (0.25, -1, 2) & (0.2389, -0.9848, 1.9458) & (0.0407, 0.2936, 0.2059)\\
100 & (0.35, -1, 2) & (0.3338, -1.0049, 1.9278) & (0.0439, 0.3151, 0.2200)\\
100 & (0.45, -1, 2) & (0.4194, -0.9836, 1.9683) & (0.0362, 0.2996, 0.2727)\\
\hline
\end{tabular}
\end{center}
\caption{\label{gauss} Empirical means and standard deviations (from $M=100$ samples of size $n$) of the estimator $\hat \theta_n = (\hat p_n, \hat \alpha_n, \hat \beta_n)$ of  $\theta_0=(p_0,\alpha_0,\beta_0)$ when $f$ is standard Gaussian.}
\end{table}
In the Table \ref{nonsym}  we propose to illustrate the sensitivity of our method with respect to the symmetry assumption by considering a symmetric case against  various  shapeless mixed distributions close to the symmetric case.
\begin{table}[hptb!]
\begin{center}
\begin{tabular}{cccc}
$n$ & $\lambda$ & Empirical means & Standard deviations\\
\hline
100 & 0.5 & (0.2302, -1.0153, 1.9420) & (0.0390, 0.2949, 0.2627)\\
100 & 0.55 & (0.2299, -1.0206, 1.9639) & (0.0418, 0.3319, 0.2693)\\
100 & 0.6 & (0.2330, -0.9703, 1.9637) & (0.0402, 0.3134, 0.2808)\\
100 & 0.65& (0.2289, -0.9938, 2.0434) & (0.0399, 0.2572, 0.2744)\\
\hline
\end{tabular}
\end{center}
\caption{\label{nonsym} Empirical means and standard deviations (from $M=100$ samples of size $n$) of the estimator $\hat \theta_n = (\hat p_n, \hat \alpha_n, \hat \beta_n)$ of  $\theta_0=(0.25,-1,2)$ when $f$ is  the pdf  of a mixture distribution $\lambda {\mathcal N} (0.5, \sqrt{2}) + (1 -  \lambda ) {\mathcal N} (-0.5 \lambda/ (1-\lambda ), \sqrt{2})$, obtained by considering $\lambda  = 0.5, 0.55, 0.6, 0.65$.}
\end{table}

\begin{table}[hptb!]
\begin{center}
\begin{tabular}{cccc}
$n$ & $(p_0,\alpha_0,\beta_0)$ & Empirical means & Standard deviations\\
\hline
100 & (0.2, 1, 5) & (0.1987, 0.9888, 5.0116) & (0.0620, 0.3127, 0.2199)\\
100 & (0.2, 1, 2) & (0.1915, 1.1103, 1.9728) & (0.0580, 0.2374, 0.2630)\\
100 & (0.2, 1, 1.5) & (0.2068, 1.0815, 1.5358) & (0.0588, 0.2267, 0.2219)\\
100 & (0.2, 1, 1.2) & (0.2092, 1.0890, 1.1871) & (0.0626, 0.2398, 0.2452)\\
\hline
\end{tabular}
\end{center}
\caption{\label{cauchy} Empirical means and standard deviations (from $M=100$ samples of size $n$) of the estimator $\hat \theta_n = (\hat p_n, \hat \alpha_n, \hat \beta_n)$ of  $\theta_0=(p_0,\alpha_0,\beta_0)$ when $f$ is standard Cauchy.}
\end{table}

\begin{table}[hptb!]
\begin{center}
\begin{tabular}{cccc}
$n$ & $(p_0,\alpha_0,\beta_0)$ & Empirical means & Standard deviations\\
\hline
100 & (0.05, -1, 2) & (0.0520, -0.9768, 2.0034) & (0.0280, 0.4276, 0.1704)\\
100 & (0.15, -1, 2) & (0.1518, -0.9765, 1.9769) & (0.0317, 0.4109, 0.1802)\\
100 & (0.25, -1, 2) & (0.2447, -1.0103, 1.9886) & (0.0290, 0.4423, 0.2056)\\
100 & (0.35, -1, 2) & (0.3432, -0.9602, 1.9407) & (0.0297, 0.4014, 0.2344)\\
100 & (0.45, -1, 2) & (0.4300, -0.9710, 1.9547) & (0.0315, 0.4114, 0.3158)\\
\hline
\end{tabular}
\end{center}
\caption{\label{laplace} Empirical means and standard deviations (from $M=100$ samples of size $n$) of the estimator $\hat \theta_n = (\hat p_n, \hat \alpha_n, \hat \beta_n)$ of  $\theta_0=(p_0,\alpha_0,\beta_0)$ when $f$ is Laplace.}
\end{table}

\begin{figure}[hptb!]
\includegraphics[width=4.5cm,height=4cm]{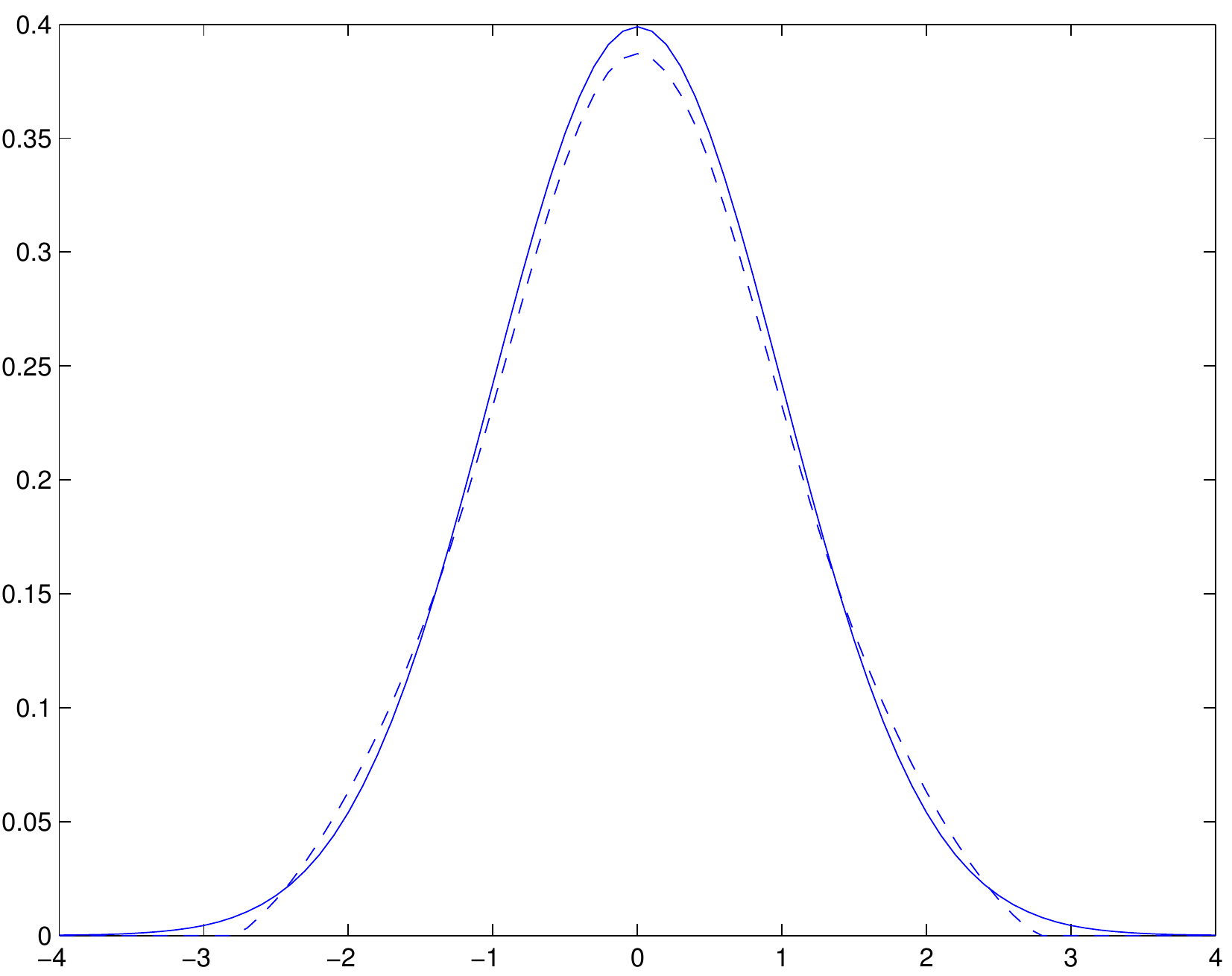}
\includegraphics[width=4.5cm,height=4cm]{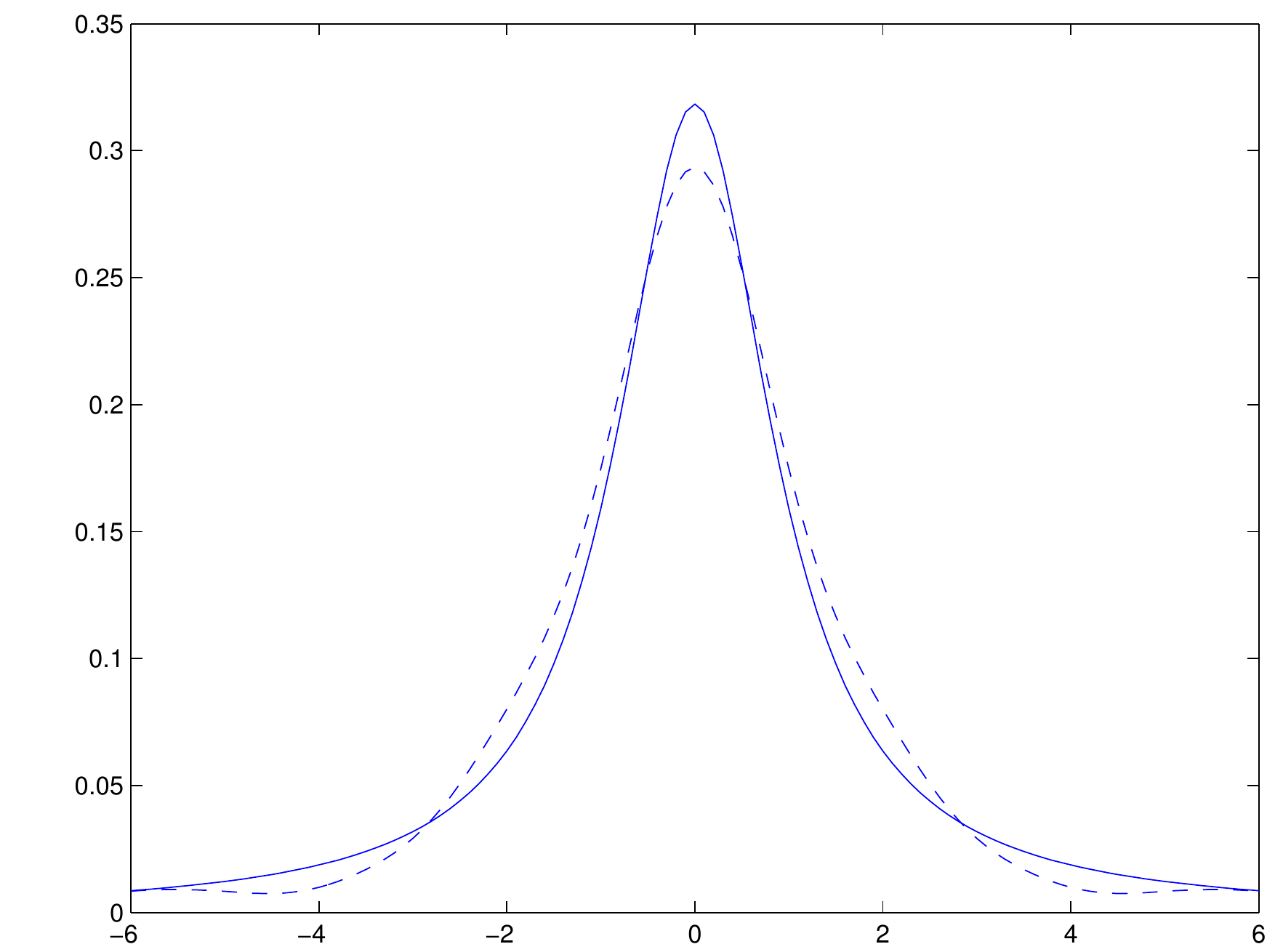}
\includegraphics[width=4.5cm,height=4cm]{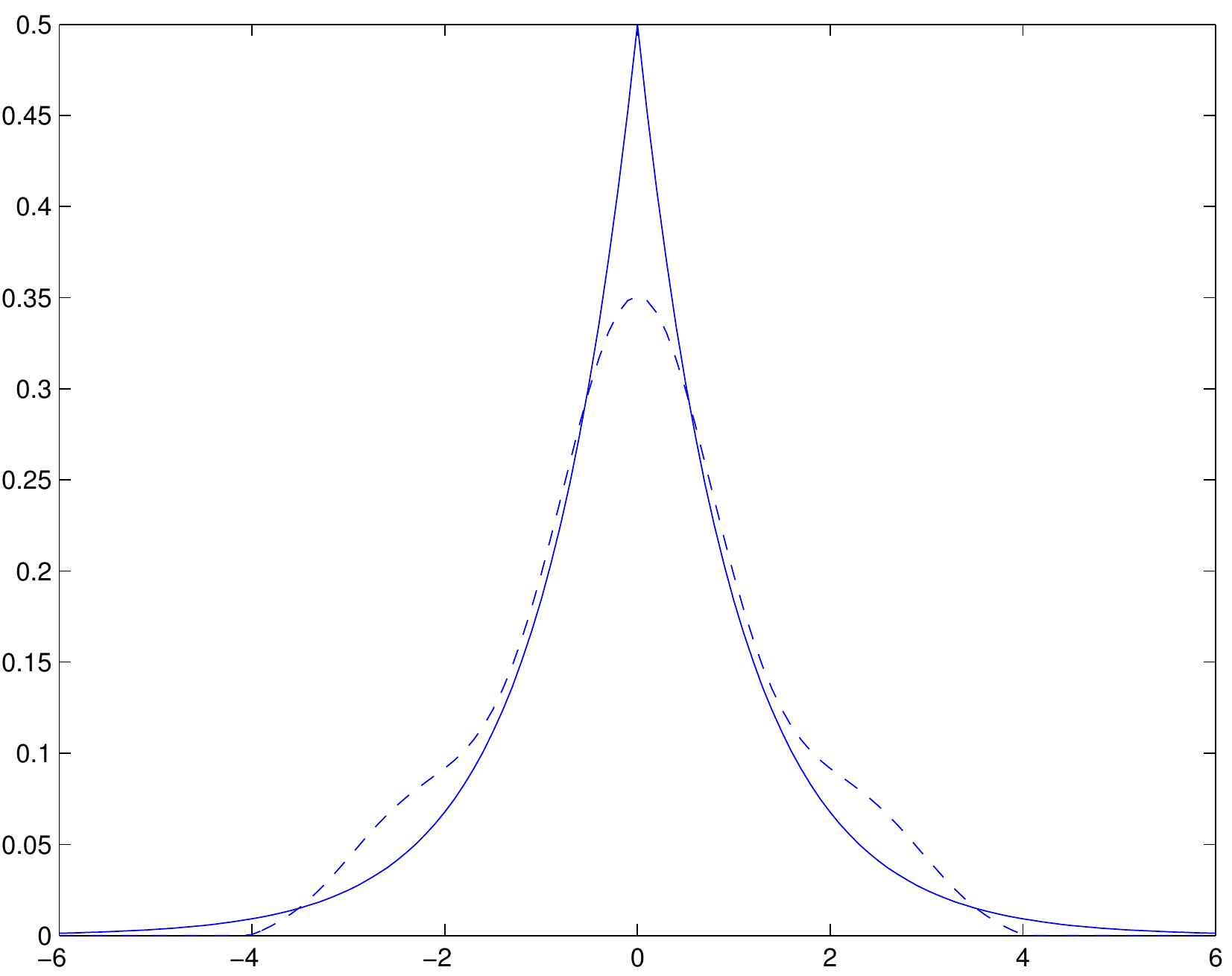}
\caption{\label{figure1} Underlying density (solid line) and kernel estimator (dashed line) for a) Gauss density, b) Cauchy density and c) Laplace density.}
\end{figure}
\noindent {\it Comments on Table 1-4}.  Comparing the rows 3 and 5 of Table \ref{gauss}  with the rows 2 and 5  of Table 2 in  BMV, it appears that our  estimator is clearly  less unstable than the estimator proposed by these authors when $f$ is the ${\mathcal N}(0,1)$ pdf. Table 2  summarizes the performance of our method in slightly shapeless  situation where $f$ is the pdf of  the  $\lambda {\mathcal N} (0.5, \sqrt{2}) + (1 -  \lambda ) {\mathcal N} (-0.5 \lambda/ (1-\lambda ), \sqrt{2})$ distribution satisfying  $\int_\R xf(x)dx=0$ and $\int_\R x^2 f(x)dx=1$, for all $\lambda \in (0,1)$. When $\lambda=0.5$ ($f$ is a symmetric bimodal pdf with mean 0 and variance equal to 1)  it is then interesting to compare the performance of our method, see  row 1 of  Table \ref{nonsym}, with its performances in the similar Gaussian case, see row 4 in Table \ref{gauss}, the noticeable fact being that the variance of $\hat\beta_n$ is smaller in the Gaussian case. When $\lambda=0.55,0.6,0.65$  the  bias of $\hat p_n$ is  badly affected  when the standard deviations of the estimators is  stable. The results provided in Table 3 seems to show that  the heavy tails of the Cauchy distribution have essentially a bad influence on the standard deviation of $\hat p_n$. Comparing Table 1 and Table 4 it appears that the peak   on the graph of the Laplace pdf helps to estimate  the parameter $p_0$  but do not work in favor of the other parameters. \\

\noindent{\it Rainfall dataset.} In this paragraph we propose to study the performances of our method when compared to the results obtained
in BMV.   We have implemented the Gauss kernel estimator with bandwidth $b_n =2 n^{-1/4}$, $n=70$,  and used in  (\ref{kernelfour}), instead of $\hat \theta_{n,-k}$,   the estimator $\hat \theta_n$. When $K$ is the Gauss kernel, we explicitly have
\begin{eqnarray*}
  f_n(x) = \frac 1{n} \sum_{k=1}^n \int_\R Q(b_n,\hat \theta_n;u)[\hat p_n \cos(u(X_k-x-\hat \alpha_n))+(1-\hat p_n) \cos (u(X_k-x-\hat \beta_n))]du,
\end{eqnarray*}
where $$Q(\theta,b;u):=\frac{1}{2\pi}\times \frac{e^{-b^2 u^2/2}}{ 2p^2-2p+1+2p(1-p)\cos(u(\alpha-\beta))}.$$
The results provided by our method are $\hat p_n=0.15$, $\hat \alpha_n=12.7$, $\hat \beta_n=38.5$ and the behavior of the functional estimators
is summarized in Figure \ref{figure2}. Before commenting  the good performances of our  estimator  $(\hat \theta_n,\tilde f_n)$  in Figure \ref{figure2}, it is crucial to notice that the reconstruction of the pdf $g$
by $g_{\hat \theta_n, f_n} (\cdot)=\hat p_n  f_n(\cdot-\hat \alpha_n)+(1-\hat p_n)  f_n(\cdot-\hat \beta_n)$ coincides  with $g_n$ itself, according to (\ref{kernelfour}-\ref{tildf}) and replacing $\hat\theta_{n,-k}$ by $\hat \theta_n$.
This basic phenomenon is illustrated in Figure \ref{figure3}.
\begin{figure}[hptb!]
\center
\includegraphics[width=8cm,height=6cm]{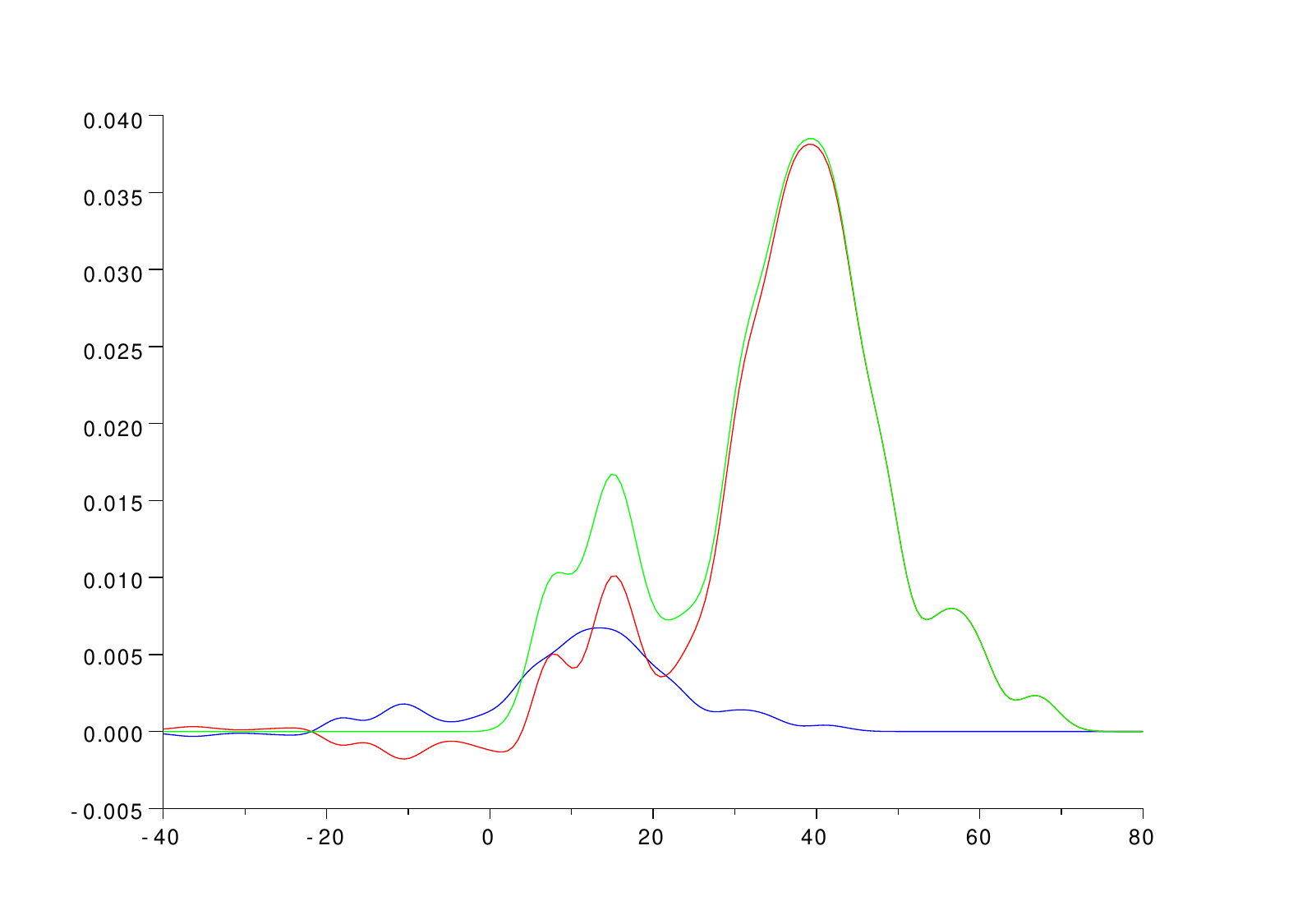}
\caption{\label{figure3} Rainfall dataset.  In blue the graph of $\hat p_n f_n(\cdot-\hat \alpha_n)$, in red the graph of $(1-\hat p_n)  f_n(\cdot-\hat \beta_n)$, in green the graph of 
$ g_{\hat \theta_n, f_n} (\cdot)=\hat p_n  f_n(\cdot-\hat \alpha_n)+(1-\hat p_n)  f_n(\cdot-\hat \beta_n)=g_n$ obtained with  $h_n=2.5$.}
\end{figure}
As mentioned in Section \ref{nonp},
the function $f_n$ is not necessarily a pdf due to its negative part (coming from the small size of $n$ and the fact  that model (\ref{model}) is not necessarily the true underlying model),
hence it is needed to regularize $f_n$ into  $\tilde f_n$  which leads to consider,  on this real dataset,  $\tilde f_n =0.9644\times f_n \mathbb{I}_{f_n\geq 0}$. This modification explains the fact the graph of $g_{\hat \theta_n,\tilde f_n}=\hat p_n \tilde  f_n(\cdot-\hat \alpha_n)+(1-\hat p_n) \tilde  f_n(\cdot-\hat \beta_n)$ does not match exactly the graph of  $g_{\hat \theta_n, f_n}=g_n$.
\begin{figure}[hptb!]
\center
\includegraphics[width=7cm,height=6cm]{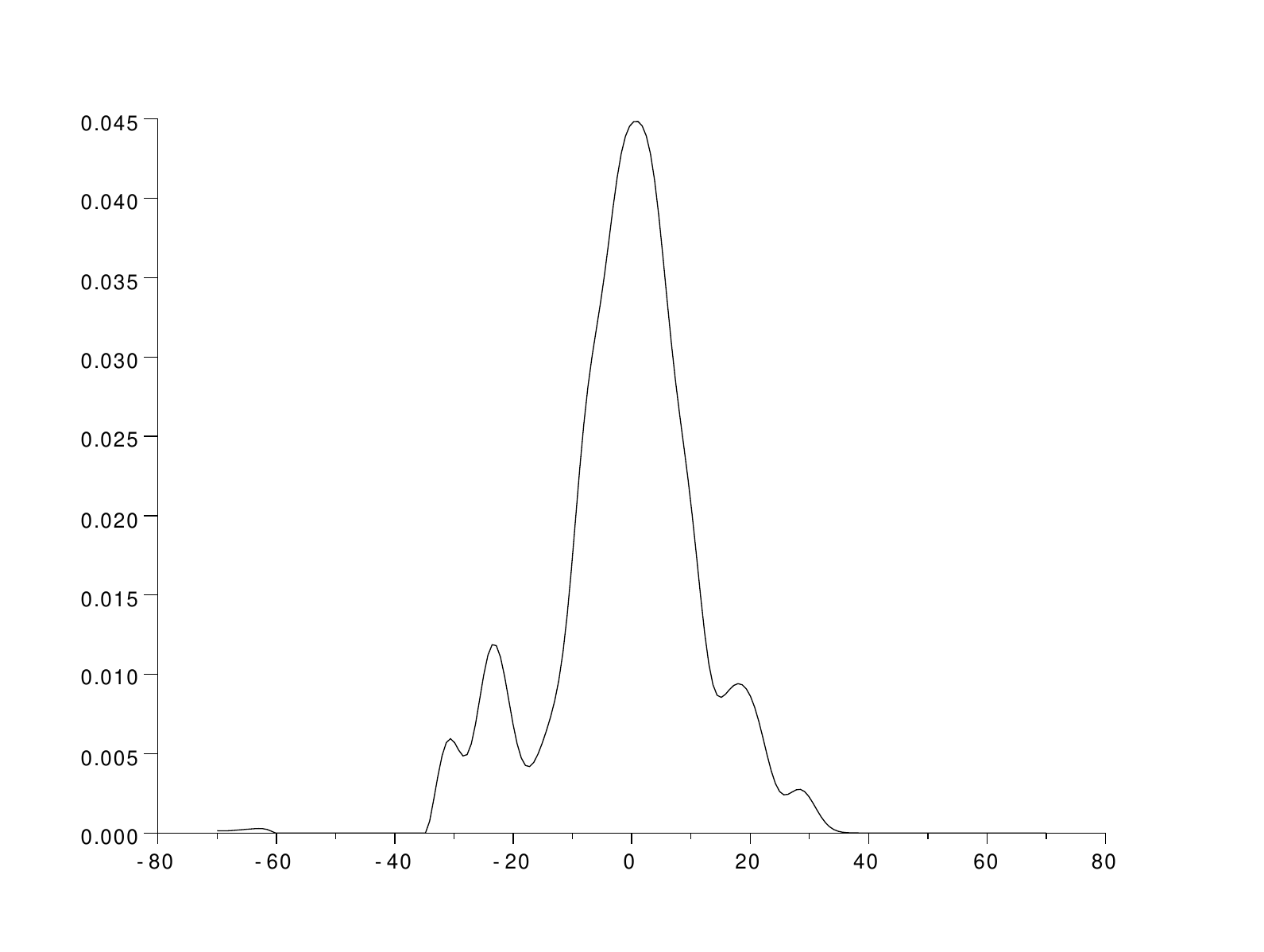}
\quad
\includegraphics[width=7cm,height=6cm]{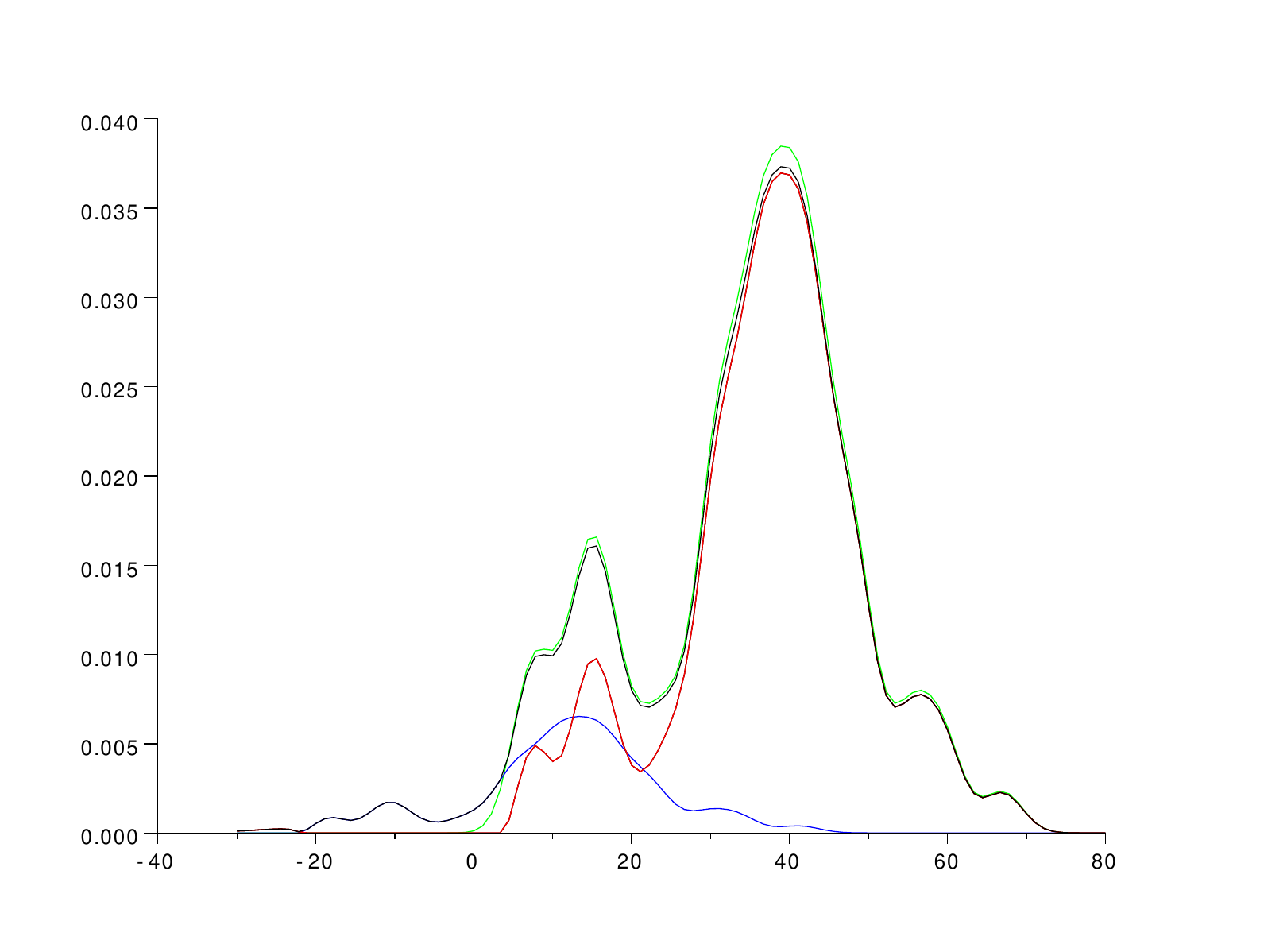}
\caption{\label{figure2} Rainfall dataset. a) Graph of $\tilde f_n$ $f$; b) In blue the graph of $\hat p_n \tilde f_n(\cdot-\hat \alpha_n)$, in red the graph of $(1-\hat p_n) \tilde f_n(\cdot-\hat \beta_n)$, in black the graph of 
$ g_{\hat \theta_n,\tilde f_n} (\cdot)=\hat p_n \tilde f_n(\cdot-\hat \alpha_n)+(1-\hat p_n) \tilde f_n(\cdot-\hat \beta_n)$, in  green  the graph of $ g_n$ obtained with  $h_n=2.5$. }
\end{figure}
Actually we observe that the graph of  $  g_{\hat \theta_n,\tilde f_n} (\cdot)$   
fits  almost perfectly the graph of $\hat g_n$ in the interval $[0,80]$, 
when it generates an extra  bump in the interval [-20,0].
Nethertheless when comparing our graphs to  the graphs obtained in  BMV (including a comparison with the two-component Gaussian  mixture model), we observe that we both  have  the extra bump issue on the intervall [-20,0],  on the other hand we better estimate  the two first bumps appearing on the graph of $g_n$ within the interval $[0,20]$.
We think that our methodological approach performs  better than the existing one, mainly  because we do not symmetrize our functional estimator
$\tilde f_n$ in order to mimic as much as possible the shape of $f_n$  (which shapeless is  precisely the reason why $g_{\hat \theta_n,\tilde f_n}=g_n$, see Figure \ref{figure3}).





\section{Auxiliary results and Proofs}

Let us use the notation $\|v\|$ for the Euclidean norm of a vector $v \in \mathbb{R}^d$
and $\|A\|^2_2 = tr(A^\top A)$ for any matrix $A$ in $\mathbb{R}^{d \times d}$.

\begin{lemma}\label{bornes}
\begin{enumerate}
\item For all $u\in \R$, we have
$$
\max\{\sup_{\theta \in \Theta} |Z_k(\theta,u)|, \sup_{\theta \in \Theta} |J(\theta,u)|\} \leq \frac 2{1-2P},
$$
for any $k$ from 1 to $n$.
\item For all $u\in \R$, we have
$$
\max\{ \sup_{\theta \in \Theta}\|\dot Z_k(\theta,u)\|,\sup_{\theta \in \Theta} \|\dot J(\theta,u)\|\} \leq \frac {4(1+|u|)}{(1-2P)^2},
$$
for any $k$ from 1 to $n$.

\item For all $u\in \R$, we have
$$
\|\ddot Z_k(\theta,u)\|_2 \leq \frac {C(1+|u|+u^2)}{(1-2P)^3},
$$
for some absolute constant $C>0$, for any $\theta \in \Theta$ and for any $k$ from 1 to $n$.

\end{enumerate}
\end{lemma}

\begin{proof}
1. It is easy to see that $|Z_j(\theta,u)| \leq 2 /|M(\theta,u)| \leq 2/(1-2P)$ and that
$$
|J(\theta,u)| \leq 2 \left| \frac{g^{\ast }(u)}{M(\theta,u)}\right|\leq \frac 2{(1-2P)}.
$$

2. We note that
$$
\dot Z_k(\theta,u)
= -\frac{e^{iuX_k}}{M^2(\theta,u)} \left(\begin{array}{c}
e^{iu\alpha }-e^{iu\beta}\\ iupe^{iu\alpha}\\ iu(1-p) e^{iu \beta}
\end{array} \right)
+\frac{e^{-iuX_k}}{M^2(\theta,-u)} \left( \begin{array}{c}
e^{-iu\alpha }-e^{-iu\beta}\\ -iupe^{-iu\alpha}\\ -iu(1-p) e^{-iu \beta}
\end{array}\right),
$$
and that
$$
E[\dot Z_k(\theta,u)] = \dot J(\theta,u)
= -\frac{g^*(u)}{M^2(\theta,u)} \left(\begin{array}{c}
e^{iu\alpha }-e^{iu\beta}\\ iupe^{iu\alpha}\\ iu(1-p) e^{iu \beta}
\end{array} \right)
+\frac{g^*(-u)}{M^2(\theta,-u)} \left( \begin{array}{c}
e^{-iu\alpha }-e^{-iu\beta}\\ -iupe^{-iu\alpha}\\ -iu(1-p) e^{-iu \beta}
\end{array}\right).
$$
We have
\begin{eqnarray*}
\| \dot J(\theta,u)\| &= & \| \frac{g^{\ast }(u)}{M^2(\theta,u)} \dot M(\theta,u)
+ \frac{g^{\ast }(-u)}{M^2(\theta,-u)} \dot M(\theta,-u)\|\\
& \leq & \frac 1{(1-2P)^2} \left(2 \left( 2^2+ p^2u^2 +(1-p)^2u^2 \right)\right)^{1/2}
\leq \frac{4(1+|u|)}{(1-2P)^2}
\end{eqnarray*}
and the same goes for $\dot Z_k(\theta,u)$.

3. We write briefly
\begin{eqnarray*}
\ddot Z_k(\theta, u) & = & -\frac{e^{iu X_k}}{M^2(\theta,u)} \ddot M(\theta,u)
+\frac{e^{-iu X_k}}{M^2(\theta,-u)} \ddot M(\theta,-u)\\
&& + 2\frac{e^{iu X_k}}{M^3(\theta,u)} \dot M(\theta,u)\cdot \dot M(\theta, u)^\top
-2\frac{e^{-iu X_k}}{M^3(\theta,-u)} \dot M(\theta,-u)\cdot \dot M(\theta,- u)^\top.
\end{eqnarray*}
We deduce our bound from above.
\end{proof}

\begin{lemma}\label{bornesLip}
\begin{enumerate}

\item For all $u\in \R$, we have
$$
\|\dot Z_k(\theta,u) - \dot Z_k(\theta',u)\| \leq \|\theta - \theta '\| \cdot \frac {C(1+ |u|+ u^2)}{(1-2P)^3},
$$
for any $\theta, \, \theta' \in \Theta$ and any $k$ from 1 to $n$.

\item For all $u\in \R$, we have
$$
\|\ddot Z_k(\theta,u) - \ddot Z_k(\theta',u)\|_2 \leq \|\theta - \theta '\| \cdot \frac {C(1+|u|+ u^2 + |u|^3)}{(1-2P)^4},
$$
for some absolute constant $C>0$, for any $\theta, \, \theta' \in \Theta$ and for any $k$ from 1 to $n$.

\end{enumerate}
\end{lemma}

\begin{proof} The proof uses a Taylor expansion and bounds from and similar to the Lemma~\ref{bornes}.
\end{proof}


\begin{proof}[Proof of Proposition~\ref{convS_n}]
It is easy to see that $E[Z_k(\theta, u)] = J(\theta,u)$.
Therefore the estimation bias is
\begin{eqnarray*}
|E(S_n(\theta)) - S(\theta)|
&=& \frac 14
\int_{|u|>1/h} \left( \frac{g^{\ast }(u)}{M(\theta,u)}
- \frac{\bar{g^{\ast}}(u)}{\bar{M}(\theta,u)} \right)^2dW(u)\\
&\leq &  \int_{|u|>1/h} \left(Im \frac{g^{\ast }(u)}{M(\theta,u)} \right)^2 dW(u)\\
&\leq & \frac{1}{(1-2p)^2} \int_{|u|>1/h} |g^\ast (u)|^2 dW(u).
\end{eqnarray*}
If we assume $f \in S(\beta,L)$, for some $\beta >0$ and $L>0$,
then
\begin{eqnarray}\label{biascrit}
|E(S_n(\theta)) - S(\theta)| &\leq & \frac{h^{2\beta}L}{(1-2P)^2}, \quad\quad h\to 0.
\end{eqnarray}

We have for the variance
\begin{eqnarray*}
&&Var(S_n(\theta))\\
 &=& \frac 1{16}E\left[\left(\frac 1{n(n-1)}
\sum_{j \not= k, j,k=1}^n \int_{|u|\leq 1/h}(Z_j(\theta,u)Z_k(\theta,u) - J^2(\theta,u))dW(u) \right)^2\right].
\end{eqnarray*}
It decomposes in $Var(S_n(\theta))=\frac 1{16} (T_n+V_n)$, where
\begin{eqnarray*}
T_n&=& E\left[\left(\frac 1{n(n-1)} \sum_{j \not= k, j,k=1}^n
\int_{|u|\leq 1/h}(Z_j(\theta,u) - J(\theta,u))(Z_k(\theta,u)-J(\theta,u))dW(u)
\right)^2 \right]
\\
V_n & =&E\left[\left(\frac 2n \sum_{k=1}^n \int_{|u|\leq 1/h} (Z_k(\theta,u)-J(\theta,u)) J(\theta,u) dW(u)
\right)^2 \right]
\end{eqnarray*}
Indeed, random variables in the previous sums are uncorrelated. Let us study the asymptotic behavior
of these terms. On the one hand,
\begin{eqnarray*}
T_n &=& \frac 1{n(n-1)} E \left[\left(
\int_{|u|\leq 1/h}(Z_1(\theta,u) - J(\theta,u))(Z_2(\theta,u)-J(\theta,u))dW(u)
\right)^2 \right]\\
&\leq & \frac 1{n(n-1)} E \left[\left(
\int_{|u|\leq 1/h} Z_1(\theta,u) Z_2(\theta,u) dW(u) \right)^2 \right]
\leq \frac {16}{(1-2P)^4 n^2},
\end{eqnarray*}
since from Lemma \ref{bornes} we have  $|Z_k(\theta,u)| \leq 2(1-2P)^{-1}$.
In addition,
\begin{eqnarray*}
V_n &=& \frac 4n  E \left[\left(\int_{|u|\leq 1/h} Z_1(\theta,u) J(\theta,u) dW(u) \right)^2 \right]
- \frac 4n \left(\int_{|u|\leq 1/h} J^2(\theta,u) dW(u) \right)^2.
\end{eqnarray*}
It is obvious that $\int_{|u|\leq 1/h} J^2(\theta,u) dW(u) \to -4 S(\theta)$ as $h\to 0$. As for the
first term, we use that $|J(\theta,u)| \leq 2(1-2P)^{-1}$
For all $u\in \R$ and $\theta \in \Theta$ and we write
\begin{eqnarray*}
E \left[\left(\int_{|u|\leq 1/h} Z_1(\theta,u) J(\theta,u) dW(u) \right)^2 \right]
&\leq & \frac 4{(1-2P)^2}.
\end{eqnarray*}
\end{proof}

\begin{lemma}\label{LipschS_n}
\begin{enumerate}[i)]
\item The function $S$ is Lipschitz over $\Theta$.
\item The empirical contrast $S_n$ defined in (\ref{estimS}) is Lipschitz over $\Theta$.
\item The empirical contrast $S_n$ defined in (\ref{estimS}) is such that $\ddot S_n$
is Lipschitz over $\Theta$.
\end{enumerate}
\end{lemma}

\begin{proof}
i) According to the mean value theorem,  we write
\begin{eqnarray*}
S(\theta) - S(\theta') & = & -\frac 14 \int_{\R} [J^2(\theta,u)-J^2(\theta',u)] dW(u)\\
 &=& -\frac 14 \int_{\R} (\theta-\theta')^{\top} \cdot \dot J^2(\theta_u,u) dW(u)\\
 &=& -\frac 12  \int_{\R} (\theta-\theta')^{\top} \cdot \dot J(\theta_u,u) J(\theta_u,u) dW(u),
\end{eqnarray*}
where for all $u\in \R$, $\theta_u$ lies in the line segment with extremities $\theta$ and $\theta'$.
By Cauchy-Schwarz inequality,
$$
|S(\theta) - S(\theta') |\leq \frac 12 \|\theta-\theta' \| \cdot \int_{\R} \| \dot J(\theta_u,u) \| \cdot |J(\theta_u,u)|dW(u).
$$
By Lemma~\ref{bornes}, $|S(\theta) - S(\theta') | \leq 4 (1-2P)^{-3} \int (1+ |u|)dW(u)\cdot \|\theta-\theta' \| $.

ii) Very similarly,
\begin{eqnarray*}
S_n(\theta) - S_n(\theta') & = & -\frac 1{4n(n-1)}  \sum_{j \not= k, j,k=1}^n
\int_{|u|\leq 1/h} (\theta-\theta')^{\top} \cdot \nabla \left( Z_k(\theta,u) Z_j(\theta,u)
\right)|_{\theta=\theta_u} dW(u)\\
&=& -\frac 1{2n(n-1)}  \sum_{j \not= k, j,k=1}^n
\int_{|u|\leq 1/h} (\theta-\theta')^{\top} \cdot \dot Z_k(\theta_u,u)
Z_j(\theta_u,u)dW(u),
\end{eqnarray*}
where for all $u\in \R$, $\theta_u$ lies in the line segment with extremities $\theta$ and $\theta'$.
Therefore
$$
|S_n(\theta) - S_n(\theta') |\leq \frac 4{(1-2P)^{3}} \|\theta-\theta' \| \cdot \int_{\R} (1+|u|) dW(u).
$$
Indeed, by Lemma~\ref{bornes}, $Z_j$ and $\dot Z_k$ have the same upper bounds as $J$ and $\dot J$, respectively.

iii) We have
$$
\ddot S_n(\theta) = \frac{-1}{2n(n-1)} \sum_{k\ne j} \int_{|u|\leq 1/h}
\left[\ddot Z_k(\theta,u) Z_j(\theta,u) + \dot Z_k(\theta,u) \dot Z_j(\theta,u)^\top \right]
dW(u).
$$
We shall bound from above as follows
\begin{eqnarray*}
\| \ddot S_n(\theta,u)-\ddot S_n(\theta ',u)\|_2
& \leq & \frac 1{2n(n-1)} \sum_{k \ne j}
\left\{ \left \| \int_{|u|\leq 1/h}(\ddot Z_k(\theta,u) -\ddot Z_k(\theta',u)) Z_j(\theta,u) dW(u)\right \|_2 \right. \\
&& + \left\|\int_{|u|\leq 1/h} \ddot Z_k(\theta',u) (Z_j(\theta,u) - Z_j(\theta',u)) dW(u) \right\|_2 \\
&& + \left\| \int_{|u|\leq 1/h} \dot Z_k(\theta,u) (\dot Z_j(\theta,u)- \dot Z_j(\theta',u))^\top dW(u)\right\|_2 \\
&& \left. + \left\| \int_{|u|\leq 1/h}(\dot Z_k(\theta,u) - \dot Z_k(\theta ',u)) \dot Z_j(\theta',u)^\top dW(u)\right\|_2\right\}.
\end{eqnarray*}
For each term in the previous sum, we use Taylor expansion and Lemmas~\ref{bornes} and~\ref{bornesLip} to get
$$
\left\| \ddot S_n(\theta,u)-\ddot S_n(\theta ',u)\right\|_2
\leq \left\|\theta - \theta' \right\| \frac{C \int (1+|u|+u^2+|u|^3) dW(u)}{(1-2P)^5},
$$
for some constant $C>0$, which finishes the proof by our Assumption A.
\end{proof}

\begin{proof}[Proof of Theorem \ref{cons}]
Our method is based on a
consistency proof for miminum contrast estimators by
Dacunha-Castelle and Duflo (1993,  p.94--96). Let us consider a
countable dense set $D$ in $\Theta$, then $\inf_{\theta\in \Theta}
S_n(\theta)=\inf_{\theta\in D}S_n(\theta) $, is a
measurable random variable. We define in addition the random
variable
$$
W(n,\xi)=\sup\left\{|S_n(\theta)-S_n(\theta')|;~(\theta,\theta')\in D^2,~ \|\theta-\theta' \|\leq \xi\right\},
$$
and recall that $S(\theta_0)=0$. Let us consider a non-empty  open
ball $B_0$ centered on $\theta_0$ such that $S$ is bounded from
below by a positive real number $2\varepsilon$ on $\Theta\backslash
B_0$. Let us  consider  us consider a sequence  $(\xi_p)_{p\geq 1}$ decreasing to zero,
and take $p$ such that there exists a covering of $\Theta\backslash
B_0$ by a finite number $\ell$ of balls $(B_i)_{1\leq i\leq \ell}$
with  centers $\theta_i\in \Theta$, $i=1,\dots,\ell$,   and  radius
less than $\xi_p$. Then, for all $\theta\in B_i$, we have
\begin{eqnarray*}
S_n(\theta)&\geq& S_n(\theta_i)-|S_n(\theta)-S_n(\theta_i)|\\
&\geq& S_n(\theta_i)-\sup_{\theta \in B_i}|S_n(\theta)-S_n(\theta_i)|,
\end{eqnarray*}
which leads to
\begin{eqnarray*}
\inf_{\theta \in \Theta\setminus B_0}S_n(\theta)\geq
\inf_{1\leq i\leq \ell}S_n(\theta_i) -W(n,\xi_p).
\end{eqnarray*}
As a  consequence we have the following events inclusions
\begin{eqnarray*}
\left\{\hat \theta_n\notin B_0\right\}
&\subseteq& \left\{\inf_{\theta\in  \Theta\setminus B_0}
S_n(\theta) < S_n(\theta_0)
\right\}\\
&\subseteq& \left\{\inf_{1\leq i\leq \ell}  S_n(\theta_i)-W(n,\xi_p)< S_n(\theta_0)\right\}\\
&\subseteq&
\left\{W(n,\xi_p)>\varepsilon\right\}\cup\left\{\inf_{1\leq i\leq
\ell}  S_n(\theta_i)-S_n(\theta_0))\leq
\varepsilon\right\}.
\end{eqnarray*}
Thus we have
\begin{eqnarray}\label{limsup1}
\left\{\hat \theta_n\notin B_0\right\}\subseteq
\left\{W(n,\xi_p)>\varepsilon\right\}\cup  \left\{\inf_{1\leq i\leq \ell} (S_n(\theta_i)-S_n(\theta_0))\leq \varepsilon\right\}.
\end{eqnarray}
By the convergence given in Proposition \ref{convS_n} we have
\begin{eqnarray*}\label{limsup2}
&&P\left( \inf_{1\leq i\leq \ell} (S_n(\theta_i)-S_n(\theta_0))\leq \varepsilon\right)\\
&&\leq 1-\prod_{i=1}^\ell (1-P(S_n(\theta_i)-S(\theta_0)\leq  \varepsilon)))\\
&&\leq 1-\prod_{i=1}^\ell (1-P(S_n(\theta_i)-S(\theta_i)+S_n(\theta_0)-S(\theta_0)\leq  \varepsilon-(S(\theta_i)-S(\theta_0)))\\
&&\leq 1-\prod_{i=1}^\ell (1-P(S_n(\theta_i)-S(\theta_i)+S_n(\theta_0)-S(\theta_0)\leq  -\varepsilon)))\\
&&\leq 1-\prod_{i=1}^\ell (1-P(|S_n(\theta_i)-S(\theta_i)|+|S_n(\theta_0)-S(\theta_0)|\geq  \varepsilon)))\\
&& \leq 1-\prod_{i=1}^\ell (1-[P(|S_n(\theta_i)-S(\theta_i)|\geq  \varepsilon))+P(|S_n(\theta_0)-S(\theta_0)|\geq  \varepsilon)])
\end{eqnarray*}
where the last term in the right hand side  of the above inequality vanishes to zero according to Proposition \ref{convS_n}.
Because $S_n$ is  Lipschitz over $\Theta$ by Lemma \ref{LipschS_n}, we have that for sufficiently large $p$,
$|S_n(\theta)-S_n(\theta')|\leq \varepsilon/2$ for all $(\theta,
\theta')$ such that $|\theta- \theta'|_2\leq \xi_{p}$, thus  $P(W(n,\xi_{p})>\varepsilon)=0$.
We just proved the  consistency in probability  of the contrast  estimator $\hat \theta_n$ defined in (\ref{estimateur}).
\end{proof}

\begin{proof}[Proof of Theorem \ref{NA}]
By a Taylor expansion of $\dot{S}_n$ around $\theta_0$, we have
\begin{eqnarray}
0 = \dot{S}_n(\hat \theta_n) = \dot{S}_n(\theta_0) + \ddot S_n(\theta^*_n)(\hat \theta_n-\theta_0)
\end{eqnarray}
where $\theta^*_n$ lies in the line segment with extremities $\hat \theta_n$ and $\theta_0$.

\noindent {\bf Step 1.} Let us prove that
\begin{eqnarray}
\dot{S}_n(\theta_0)=\frac{-1}{2n(n-1)}\sum_{j \not= k, j,k=1}^n \int_{|u|\leq 1/h} \dot Z_k(\theta,u){Z}_j(\theta,u)dW(u)
\end{eqnarray}
is asymptotically normal $\sqrt{n}\dot S_n(\theta_0)\stackrel{d} \to N(0,V)$, in distribution.

Indeed, $\dot S(\theta_0)=0$ and $ J (\theta_0,u) = 0$ for all $u \in \R$ imply that
$$
E[\dot S_n(\theta_0)] = -\frac{1}{2} \int_{|u| \leq 1/h} \dot J(\theta_0,u) J(\theta_0,u) dW(u) = 0.
$$
Therefore we decompose
\begin{eqnarray*}
\dot S_n(\theta_0) 
&=& \frac{-1}{2n(n-1)}\sum_{j \not= k, j,k=1}^n \int_{|u|\leq 1/h}
\dot Z_k(\theta_0,u){Z}_j(\theta_0,u) dW(u)\\
&=& \frac{-1}{2n(n-1)}\sum_{j \not= k, j,k=1}^n \int_{|u|\leq 1/h}
\left[\dot Z_k(\theta_0,u)- \dot J(\theta_0,u) \right] {Z}_j(\theta_0,u) dW(u)\\
&& -\frac 1{2n} \sum_{k=1}^n \int_{|u|\leq 1/h}
{Z}_k(\theta_0,u) \dot J(\theta_0,u) dW(u) = : A_n+B_n.
\end{eqnarray*}

We shall see that $\sqrt{n}B_n$ gives the dominant behaviour in the limit in distribution. Indeed,
\begin{eqnarray*}
&&\|n \mbox{Var}(A_n) \| \\
&\leq & \frac 1{4(n-1)}\left\| E\left[\left(
\int_{|u|\leq 1/h} \dot Z_1(\theta_0,u) Z_2(\theta_0,u) dW(u)
\right)
\left(
\int_{|u|\leq 1/h} \dot Z_1(\theta_0,u) Z_2(\theta_0,u) dW(u)
\right)^\top
\right]\right\|\\
&\leq & \frac{C}{(1-2P)^6 n} \left( \int (1+|u|)dW(u) \right)^2 = o(1).
\end{eqnarray*}
The asymptotic behaviour of the distribution of $\sqrt{n} B_n$ is obtained by noticing that
\begin{eqnarray*}
\sqrt{n} B_n &=& C_n + D_n, \quad \mbox{where}\\
C_n &=& \frac 1{2\sqrt{n}} \sum_{k=1}^n U_n(\theta_0), \\
D_n &=& \frac 1{2\sqrt{n}}\sum_{k=1}^n[U_{k,n}(\theta_0)-U_k(\theta_0)],
\end{eqnarray*}
and
$$U_{k,n}(\theta_0) = \int_{|u|\leq 1/h}
{Z}_k(\theta_0,u) \dot J(\theta_0,u)dW(u) $$
is a centered variable which depends on $n$ via $h$,
\begin{eqnarray*}
U_{k}(\theta_0) &=& \int_{\R}  {Z}_k(\theta_0,u)  \dot J(\theta_0,u)dW(u)
\end{eqnarray*}
is a centered variable not dependent on $n$.
Note that $D_n = o_P(1)$ as
\begin{eqnarray*}
&& \frac 1n \left\|Var\left(\sum_{k=1}^n (U_{k,n}(\theta_0)-U_k(\theta_0) \right) \right\|\\
& \leq & \left\|E\left( \int_{|u|>1/h} Z_1(\theta_0,u) \dot J(\theta_0,u) dW(u) \cdot
\int_{|u|>1/h} Z_1(\theta_0,u) \dot J(\theta_0,u)^\top dW(u) \right)\right\|\\
& \leq & \left\| \int_{|u|>1/h} \frac 2{(1-2P)^2} \left(
\begin{array}{c}
2 \\ |u| \\ |u|
\end{array}\right) dW(u) \cdot
\int_{|u|>1/h} \frac 2{(1-2P)^2}  \left( \begin{array}{c}
2 \\ |u| \\ |u|
\end{array}\right)^\top dW(u) \right\|\\
& \leq & \frac {o(1)}{(1-2P)^4},
\end{eqnarray*}
as $h \to 0$, since every integral in the finite sum tends to 0 when $h \to 0$.
In a standard way, $C_n$ satisfies the following central limit theorem:
\begin{eqnarray}
\frac 1{2\sqrt{n}} \sum_{k=1}^n U_k(\theta_0)\stackrel{\mathcal L}{\longrightarrow} {\mathcal{N}}(0,V),\quad \quad n\rightarrow \infty,
\end{eqnarray}
where $V$ denotes covariance matrix of $U_1(\theta_0)$ which is equal to $1/4 \cdot E(U_1(\theta_0)U_1(\theta_0)^T)$ (and cannot be explicited due to the integral nature of the terms).

\noindent {\bf Step 2.} Let us prove that
\begin{eqnarray}
\ddot S_n(\theta^*_n)\stackrel{\mathbb P}{\longrightarrow} \mathcal{I}(\theta_0), \quad \quad n\rightarrow \infty.
\end{eqnarray}
where $\mathcal{I}=\mathcal{I}(\theta_0)=-\frac 12 \int \dot J (\theta_0,u) \dot J ^\top (\theta_0,u) dW(u)$.

We start by writing the triangular inequality
$$
\|\ddot S_n(\theta_n^*) - \mathcal{I}\|\leq \|\ddot S_n(\theta_n^*) - \ddot S_n(\theta_0)\| +\|\ddot S_n(\theta_0) - \mathcal{I}\|.
$$
Then we use the Lipschitz property of $\ddot S_n$, Lemma~\ref{bornesLip}, and the convergence in probability of $\hat \theta_n$ to $\theta_0$. Finally, we compute the limit of $\ddot S_n(\theta_0)$. Indeed
\begin{eqnarray*}
E(\ddot S_n(\theta_0))&=& -\frac 12 \int_{|u|\leq 1/h} (\ddot J(\theta_0,u) J(\theta_0,u) + \dot J(\theta_0,u) \cdot J(\theta_0,u)^\top) dW(u)\\
&=& -\frac 12 \int_{|u|\leq 1/h}\dot J(\theta_0,u) \cdot J(\theta_0,u)^\top dW(u),
\end{eqnarray*}
as $J(\theta_0,u) = 0$. We see that $E(\ddot S_n(\theta_0)) \to  \mathcal{I}(\theta_0)$, as $h \to 0$.
\end{proof}


\begin{proof}[Proof of the Theorem~\ref{nonparam}]

Note first that
\begin{eqnarray*}
E(f_n(x) )&=& E\left(\frac 1{2\pi}\int e^{-iux} \frac 1n \sum_{k=1}^n \frac{e^{iuX_k} K^*(b_n u)}{M(\hat\theta_{n,-k}, u)} du \right) \\
& = & \frac 1{2\pi}\int e^{-iux} g^*(u) K^*(b_n u) E\left( \frac 1{M(\hat\theta_{n,-1}, u)} \right)du.
\end{eqnarray*}
Recall that $\sup_{\theta \in \Theta} |M(\theta,u)| \geq 1-2P$, which means that
$ E(M^{-1}(\hat\theta_{n,-1}, u))\leq (1-2P)^{-1}$.

Let us write the usual bias-variance decomposition. For the bias, we have
\begin{eqnarray*}
E(f_n(x)) - f(x) &=&
\frac 1{2\pi}\int e^{-iux} g^*(u) \left( K^*(b_n u) E\left( \frac 1{M(\hat\theta_{n,-1}, u)}\right)
- \frac 1{M(\theta_0,u)} \right)du\\
&=& \frac 1{2\pi}\int e^{-iux} g^*(u)  K^*(b_n u) \left(E\left( \frac 1{M(\hat\theta_{n,-1}, u)}\right)
- \frac 1{M(\theta_0,u)} \right)du\\
&& + \frac 1{2\pi}\int e^{-iux} \frac{g^*(u)}{M(\theta_0,u)} ( K^*(b_n u)-1) du.
\end{eqnarray*}
Next, we use the facts that $|\sup_u K^*(u)| \leq 1$ and that the support of $K^*(b_n u)$ is included in $\{u: |u|\geq 1/b_n \}$ and get
\begin{eqnarray*}
|E(f_n(x)) - f(x)|&\leq &
\frac 1{2\pi} \left( \int |g^*(u)| | E(M^{-1}(\hat \theta_{n,-1}),u)-M^{-1}(\theta_0,u)| du \right.\\
&& \left. + \frac 1{1-2P} \int_{|u|\geq 1/b_n} |g^*(u)| du
\right) \\
&= & O(\frac 1{\sqrt{n}}) + O(1) \frac{b_n^{\beta - 1/2}}{1-2P}.
\end{eqnarray*}

For the variance, we write
\begin{eqnarray*}
&&Var(f_n(x)) \\
&= & E\left[ \left(\frac 1{2\pi n} \sum_{k=1}^n \int e^{-iux}K^*(b_n u) \left( \frac{e^{iuX_k} }{M(\hat\theta_{n,-k}, u)} - g^*(u) E\left(\frac 1{M(\hat \theta_{n,-1},u)}\right)\right) du  \right)^2\right]\\
&\leq & \frac 1{4\pi^2 n}  E\left[ E\left[\left(\int K^*(b_n u) \frac{e^{iuX_1} }{M(\hat\theta_{n,-1}, u)} du
\right)^2 / X_2,...,X_n \right]\right]\\
&\leq & \frac 1{4\pi^2 n}  E\left[ \left(\int K^*(b_n u) \frac{g^*(u) }{M(\hat\theta_{n,-1}, u)} du
\right)^2\right]
\leq \frac{\|K^*\|_2^2\|g^*\|_2^2}{4\pi^2 (1-2P)^2 n b_n}.
\end{eqnarray*}
Therefore, for $b_n= c n^{-(\beta-1/2)/(2\beta)}$ we get the upper bounds in our theorem.
\end{proof}


\noindent Pierre Vandekerkhove, D\'epartement d'Analyse et de  Math\'ematiques Appliqu\'ees, Universit\'e  Paris-Est Marne-la-Vall\'ee, 5 Bd Descartes, Champs-sur-Marne, 77454 Marne-la-Vall\'ee Cedex 2.\\
\noindent Email: pierre.vandek@univ-mlv.fr

\end{document}